\documentclass[a4paper,11pt]{amsart}
\usepackage{amssymb,amsmath,mathrsfs,url,float,hyperref}   
\usepackage[numbers, sort&compress]{natbib}
\usepackage[inner=3cm,outer=3cm,top=3cm,bottom=3cm,includehead,includefoot]{geometry}
\hypersetup{citecolor=red, linkcolor=blue, colorlinks=true}
\newtheorem{thm}{Theorem}[section]
\newtheorem{prop} [thm]{Proposition}
\newtheorem{cor} [thm]{Corollary}
 \newtheorem{lemma} [thm]{Lemma}
\theoremstyle{definition}

\newtheorem{rem}[thm]{Remark}

\renewcommand\leq{\leqslant} 
\renewcommand\geq{\geqslant}

\DeclareMathOperator{\GL}{GL}
\DeclareMathOperator{\GammaL}{\Gamma L}
\DeclareMathOperator{\SL}{SL}

\title[Base sizes of imprimitive linear groups]{Base sizes of imprimitive linear groups and orbits of general linear groups on spanning tuples}

\author{Joanna B. Fawcett, Cheryl E. Praeger}

\address{
Centre for the Mathematics of Symmetry and Computation\\
School of Mathematics and Statistics\\
The University of Western Australia\\
35 Stirling Highway, Crawley, WA 6009, Australia. Email: \texttt{\{joanna.fawcett, cheryl.praeger$^\dag$\}@uwa.edu.au}\\
\newline $^\dag$ Also affiliated with King Abdulaziz University, Jeddah, Saudi Arabia.
}

\keywords{permutation group, base size, general linear group, imprimitive linear group, spanning sequence}

\begin{document}
 \maketitle
 \vspace{-0.2cm}
 \begin{abstract}
For a subgroup $L$ of the symmetric group $S_\ell$, we determine the minimal base size of  $\GL_d(q)\wr L$ acting on $V_d(q)^\ell$ as an imprimitive linear group. This is achieved by computing the number of orbits of $\GL_d(q)$ on spanning $m$-tuples, which turns out to be the number of $d$-dimensional subspaces of $V_m(q)$. We then use these results to prove that for certain families of subgroups $L$, the  affine groups whose stabilisers are large subgroups of $\GL_d(q)\wr L$ satisfy a conjecture of Pyber concerning bases.
 \end{abstract}

\section{Introduction}
 
Bases are a fundamental tool in permutation group theory and are used extensively in computational group theory (cf. \cite{Ser2003}). For  a permutation group  $G$ on $\Omega$, a \textit{base}  is a subset $B$ of $\Omega$ with the property that only the identity of $G$ fixes every point of $B$.  The \textit{base size} of $G$ on $\Omega$, denoted by $b_\Omega(G)$ or $b(G)$, is the minimal size of a base for $G$. In this paper, we study the base sizes of imprimitive linear groups.

Let $V_d(q)$ denote a $d$-dimensional vector space over the finite field $\mathbb{F}_q$. For any positive integer $\ell$ and subgroup $L$ of the symmetric group $S_\ell$, the wreath product $\GL_d(q)\wr L$  acts naturally on $V_d(q)^\ell$ as an imprimitive linear group (cf.\ Section \ref{s:prelim}). In our first result, we determine the base size of $\GL_d(q)\wr L$  in terms of the \textit{distinguishing number} of $L$  on $[\ell]:=\{1,\ldots,\ell\}$; this latter quantity, denoted by $d_{[\ell]}(L)$ or $d(L)$, is the minimal number of parts in a partition of $[\ell]$ for which only the identity of $L$ fixes every part.

\begin{thm} 
\label{thm:base size}
Let $d$ and $\ell$ be  positive integers and $q$ a prime power. Let $V:=V_d(q)^\ell$. For  $L\leq S_\ell$, 
\begin{align*}
b_V(\GL_d(q)\wr L)&=d+\min\left\{ s :\tbinom{d+s}{d}_q\geq d_{[\ell]}(L)\right\}\\
&=d+\left\lceil\tfrac{ \log d_{[\ell]}(L)}{d\log q} \right\rceil +c,
\end{align*}
where $c=-1$ or $0$.
\end{thm}

Theorem \ref{thm:base size} gives an upper bound on the base size of any irreducible imprimitive linear group, for we may view such a group $H$ as  a subgroup of $\GL_d(q)\wr L$ where $L$  is the transitive group induced by $H$ on the   $d$-dimensional $\mathbb{F}_q$-vector spaces of a direct sum decomposition preserved by $H$  (cf. Lemma \ref{lemma:irred}). 
In fact, if the decomposition preserved by $H$ is as coarse as possible, then the group $L$ is primitive (cf. Lemma \ref{lemma:irred}), and if $L$ is not the full symmetric or alternating group, then  $d(L)\leq 4$ by \cite{Ser1997,Dol2000,CamNeuSax1984} (cf. Theorem \ref{thm:dist}), so we obtain the following consequence of Theorem \ref{thm:base size}.

\begin{cor}
\label{cor:base size}
Let $V$ be a finite-dimensional $\mathbb{F}_q$-vector space and $H\leq \GL(V)$ where $H$  is irreducible and imprimitive. Let $V=V_1\oplus \cdots\oplus V_\ell$ be a decomposition preserved by $H$, chosen so that  $\ell$ is minimal subject to $\ell \geq 2$. If  the permutation group induced by $H$ on $\{V_1,\ldots,V_\ell\}$ is not $S_\ell$ or $A_\ell$,  then $b_V(H)\leq \dim_{\mathbb{F}_q}(V_1)+1$.
\end{cor}

Theorem \ref{thm:base size} is proved using a result of Bailey and Cameron \cite{BaiCam2011} that describes the base size of  $\GL_d(q)\wr L$  in terms of the number of orbits of $\GL_d(q)$ on \textit{spanning $m$-tuples}, which are $m$-tuples $(v_1,\ldots, v_m)\in V_d(q)^m$ such that $v_1,\ldots,v_m$ span $V_d(q)$ (such sequences are referred to as ordered multi-bases in the more general setting of \cite{BaiCam2011}).  

In our next result, we determine the number of orbits of $\GL_d(q)$ in its natural action on the set of spanning $m$-tuples; here we find another interpretation for the Gaussian binomial coefficient $\tbinom{m}{d}_q$, which equals, for instance,  the number of $d$-dimensional subspaces of an $m$-dimensional $\mathbb{F}_q$-vector space.

\begin{thm}
\label{thm:orbits}
Let $d$ and $m$ be positive integers and $q$ a prime power. Then $\GL_d(q)$ has exactly $\tbinom{m}{d}_q$  orbits  on the set of spanning $m$-tuples in $V_d(q)^m$.
\end{thm}

As an application, we use Theorem \ref{thm:base size} to prove a conjecture of Pyber for certain affine groups. Given a permutation group $G$  of degree $n$, there is a trivial lower bound on $b(G)$, namely $\log{|G|}/\log{n}$, as each element of $G$ is uniquely determined by its action on a base. Pyber conjectured in \cite{Pyb1993} that there exists an absolute constant $C$ such that $b(G)\leq C\log{|G|}/\log{n}$ for every primitive permutation group $G$ of degree $n$. This conjecture has now been verified for all non-affine groups \cite{Ben05,LieSha1999,Faw2013,BurSer15}, as well as for affine groups that are soluble \cite{Ser1996} or coprime \cite{GluMag1998}, or those whose stabilisers are primitive  linear groups \cite{LieSha2002,LieSha2014}. 

Thus the remaining open case for Pyber's conjecture consists of affine groups whose stabilisers are imprimitive linear groups. Here $G=V: G_0$ and acts on $V$, where $V$ is an $\mathbb{F}_p$-vector space for some prime $p$ and the stabiliser $G_0$ is an irreducible imprimitive subgroup of $\GL(V)$. We focus on (not necessarily primitive) affine groups  $G$ for which $G_0$ is a large subgroup of $\GL_d(q)\wr L$ where $L$ is one of several families of groups.

\begin{thm}
\label{thm:Pyber}
Let $d$ and $\ell$ be positive integers and $q$ a prime power.  Let $L\leq S_\ell$ and $V:=V_d(q)^\ell$. Let $G_0$ be a group for which $\SL_d(q)^\ell\leq G_0\leq\GL_d(q)\wr L$ and suppose that $G_0$ induces the group $L$ on the $\ell$ factors of the decomposition of $V$. Let $G:=V:G_0$. Suppose that one of the following holds.
\begin{itemize}
\item[(i)] $d_{[\ell]}(L)\leq c$ where $c$ is an absolute constant.
\item[(ii)] $L$ acts primitively on $[\ell]$.
\item[(iii)] $L$ acts semiregularly on $[\ell]$.
\item[(iv)] $L=S_m\wr S_r$ in its imprimitive action on $[\ell]$ where  $\ell=mr$ and $m,r\geq 2$. 
\end{itemize}
Then there exists an absolute constant $C$ such that 
$$b_V(G)\leq C \frac{\log{|G|}}{\log{n}}+C+2,$$
where $n=q^{d\ell}$ is the degree of $G$. Moreover, we can take $C=\max\{2,\tfrac{\log{2c}}{\log{2}}\}$  when (i) holds, $C=3$ when (ii) holds, and $C=2$ when (iii) or (iv) hold.
\end{thm}

More precise estimates for $C$ may be deduced from the proof of Theorem \ref{thm:Pyber} in Section~\ref{s:Pyber}. For $\ell\geq 2$, the affine groups $V:G_0$ of Theorem \ref{thm:Pyber} are primitive precisely when $L$ is transitive and $(d,q)\neq (1,2)$ (cf.\ Lemma \ref{lemma:irred}), so Theorem \ref{thm:Pyber} includes genuine primitive affine groups.

Condition (i) of Theorem \ref{thm:Pyber} holds for many permutation groups, including those  that are primitive but not $S_\ell$ or $A_\ell$, as previously mentioned. Indeed, Dolfi proves in \cite{Dol2000} that if $L$ is a (not necessarily transitive) permutation group on $[\ell]$ for which  no primitive constituent  contains $A_\ell$, 
then $d(L)\leq 5$. 

It would be interesting to prove Theorem \ref{thm:Pyber} for all transitive subgroups $L$ of $S_\ell$. In order to apply  our methods in general (cf. Lemma \ref{L trick}), it would suffice  to prove that $\log d(L)\leq (C/\ell)\log{|L|}+(C-1)\log{2}$ for some absolute constant $C$ where $C>1$.

\begin{rem}
As stated, our main results on base sizes only apply to linear groups, but each result can be interpreted for the appropriate semilinear groups, for  if $H\leq \GammaL(V)$, then $b_V(H)\leq b_V(H\cap \GL(V))+1$. To see this, let $B$ be a base for $H\cap \GL(V)$. Choose a primitive element $\zeta\in\mathbb{F}_q$ and a non-zero vector $v\in B$. Then $B\cup\{\zeta v\}$ is a base for $H$.
\end{rem}

Both Corollary \ref{cor:base size} and Theorem \ref{thm:Pyber}(ii)  depend on the classification of the finite simple groups, for their proofs rely on the result referred to above concerning the distinguishing numbers of primitive permutation groups.

This paper is organised as follows. In Section \ref{s:prelim}, we state some definitions, notation and preliminary results. In Section \ref{s:orbits}, we prove Theorem \ref{thm:orbits}; in Section \ref{s:basesize}, we prove Theorem \ref{thm:base size} and Corollary \ref{cor:base size}; and in Section \ref{s:Pyber}, we prove Theorem \ref{thm:Pyber}.

\section{Preliminaries}
\label{s:prelim}

 Let $H$ and $K$ be  groups. We denote a semidirect product of $H$ and $K$ (in which $H$ is normal) by $H:K$. If $K$ acts on $[\ell]:=\{1,\ldots,\ell\}$, then $K$ acts on $H^\ell$ by permuting coordinates, and this action defines the \textit{wreath product} $H^\ell:K$, which we denote by $H\wr K$. Suppose in addition that $H$ acts on $\Omega$. Now $H\wr K$ has two natural actions. One is the \textit{product action} on $\Omega^\ell$, in which $K$ acts by permuting coordinates and  $(h_1,\ldots,h_\ell)\in H^\ell$ maps 
$(\alpha_1,\ldots,\alpha_\ell)\in \Omega^\ell$ to $(\alpha_1^{h_1},\ldots,\alpha_\ell^{h_\ell})$. The other is the \textit{imprimitive action} on $\Omega\times [\ell]$, in which $(h_1,\ldots,h_\ell)k\in H\wr K$ maps $(\alpha,i)\in \Omega\times [\ell]$ to $(\alpha^{h_i},i^k)$; if $H$ and $K$ are transitive on $\Omega$ and $[\ell]$ respectively, then $H\wr K$ acts transitively on $\Omega \times [\ell]$ with  blocks of imprimitivity  $\{(\alpha,i):\alpha\in \Omega\}$ for $i\in [\ell]$.

Let $m$ be a positive integer. 
For a group $H$ acting on $\Omega$, an \textit{ordered multi-base of length $m$} is an $m$-tuple $(\alpha_1,\ldots,\alpha_m)\in \Omega^m$ for which $\{\alpha_1,\ldots,\alpha_m\}$ is a base for $H$. Now $H$ acts naturally on $\Omega^m$ and preserves  the set of ordered multi-bases. The following is a result of Bailey and Cameron \cite[Theorem 2.13]{BaiCam2011}.

\begin{prop}[\cite{BaiCam2011}]
\label{prop:wreath}
Let $\ell$ be a positive integer, and let $H$ and $K$ be permutation groups on $\Omega$ and $[\ell]$ respectively. Then   $H\wr K$ has a base of size $m$ under the product action if and only if the number of orbits of $H$ on ordered multi-bases of length $m$ is at least $d_{[\ell]}(K)$.
\end{prop}

In \cite{Cha06}, Chan determines the distinguishing number of a wreath product $H\wr K$ in its imprimitive action. In particular, she proves that for positive integers $m$ and $r$, the distinguishing number of $S_m\wr S_r$ on $[m]\times [r]$ is  the minimum  $d$ such that $\tbinom{d}{m}$ is at least $r$. The following observation is a simple consequence of this result.

\begin{lemma}
\label{lemma:chan}
Let $m$ and $r$ be positive integers, and let $\Delta:=[m]\times [r]$. Then $$
d_\Delta(S_m\wr S_r)\leq \left\lceil mr^{1/m}\right\rceil.$$
\end{lemma}

\begin{proof}
By \cite[Corollary 2.4]{Cha06}, $d_\Delta(S_m\wr S_r)=\min\{d: \tbinom{d}{m}\ \geq r\}$. In particular, $d\geq m$.  
Now
$$\left( \frac{d}{m}\right )^m\leq \frac{d}{m}\frac{d-1}{m-1}\cdots\frac{d-m+1}{1}=\binom{d}{m},
$$
so $d_\Delta(S_m\wr S_r)\leq \min\{ d :(d/m)^m\geq r\}=\left\lceil mr^{1/m} \right\rceil$.
\end{proof}

In contrast, the distinguishing numbers of most primitive permutation groups  are very small. Cameron, Neumann and Saxl \cite{CamNeuSax1984} proved   that all but finitely many primitive subgroups of $S_\ell$ not containing $A_\ell$ have distinguishing number $2$ (using different terminology), after which  Seress \cite{Ser1997} classified the exceptions. Dolfi \cite[Lemma 1]{Dol2000} then proved that the distinguishing numbers of the exceptions are at most $4$. We state this result here for convenience.

\begin{thm}[\cite{CamNeuSax1984,Ser1997,Dol2000}]
\label{thm:dist}
 If $L$ is a primitive subgroup of $S_\ell$ not containing $A_\ell$, then $d_{[\ell]}(L)\leq 4$.
\end{thm}

Let $V$ be a finite-dimensional $\mathbb{F}_q$-vector space where  $q$ is a power of a prime $p$. A subgroup $H$ of $\GL(V)$ is \textit{irreducible} if it does not preserve any proper non-zero subspaces of $V$, and \textit{imprimitive} if it preserves a decomposition  $V=V_1\oplus\cdots \oplus V_\ell$ where $V_i$ is an $\mathbb{F}_q$-subspace of $V$ for $1\leq i\leq \ell$. If $L$ is the group induced by $H$ on $\{V_1,\ldots,V_\ell\}$ and $\dim_{\mathbb{F}_q}(V_i)=d$ for $1\leq i\leq \ell$, then we may assume that $H\leq \GL_d(q)\wr L$; in particular, this occurs whenever $L$ is transitive. Note that the action of the imprimitive linear group $\GL_d(q)\wr L$ on $V_d(q)^\ell$ is precisely the product action defined above. 

An \textit{affine} group with socle $V$ and stabiliser $H$ is the group  $G=V:H$ arising from the natural action of $H$ on $V$. Note that $H$  is the stabiliser of the $0$ vector. The action of $G$ on $V$ is primitive precisely when $H$ is an irreducible subgroup of $\GL(V)$ with $V$  viewed as an $\mathbb{F}_p$-vector space.

The following is a collection of basic results concerning imprimitive linear groups.
 
\begin{lemma}
\label{lemma:irred}
Let $V$ be a finite-dimensional $\mathbb{F}_q$-vector space and $H\leq \GL(V)$ where $H$ is imprimitive. Let $V=V_1\oplus \cdots \oplus V_\ell$ be a decomposition preserved by $H$ where $\ell\geq 2$, and let $L$ be the subgroup of $S_\ell$ induced by $H$ on $\{V_1,\ldots,V_\ell\}$. Let $d:=\dim_{\mathbb{F}_q}(V_1)$. 
\begin{itemize}
\item[(i)] If $H$ is irreducible, then $L$ is transitive and  $(d,q)\neq (1,2)$. 
\item[(ii)] If the decomposition of $V$ is chosen so that $\ell$ is minimal, then $L$ is primitive.
\item[(iii)] If $V_i=V_d(q)$ for $1\leq i\leq \ell$ and $S\leq \GL_d(q)$ such that $S^\ell\leq H\leq \GL_d(q)\wr L$ and $S$ is transitive on $V_d(q)\setminus \{0\}$, then  $V:H$  is primitive if and only if $L$ is transitive and $(d,q)\neq (1,2)$.
\end{itemize}
\end{lemma}

\begin{proof}
(i) If $H$ is irreducible and $I$ is an orbit of $L$ on $[\ell]$, then $\oplus_{i\in I} V_i$ is a subspace of $V$ preserved by $H$, so  $L$ is transitive. If also $(d,q)=(1,2)$, then the set of vectors in $V\simeq \mathbb{F}_2^\ell$ with an even number of non-zero entries is preserved by $H$, a contradiction.

(ii) If $L$ is not primitive, then $H$ preserves a coarser decomposition of $V$.

(iii) If $V:H$ is primitive, then $H$ is irreducible and we may apply (i). Conversely, suppose that $L$ is transitive and $(d,q)\neq (1,2)$. Let $U$ be a non-zero $\mathbb{F}_p$-subspace of $V$ preserved by $H$ where $p$ is the characteristic of $\mathbb{F}_q$.  Let $0\neq u=(u_1,\ldots,u_\ell)\in U$. Without loss of generality, we may assume that $u_1\neq 0$.    Since $(d,q)\neq (1,2)$, there exists $v_1\in V_1\setminus \{0,u_1\}$. Since $H$ contains a subgroup $S$ that  is transitive on $V_1\setminus \{0\}$ and fixes $V_i$ pointwise for $2\leq i\leq \ell$, the vector $v:=(u_1-v_1,u_2,\ldots,u_\ell)\in U$. Thus $(v_1,0,\ldots,0)=u-v\in U$, and it follows as above that $(w_1,0,\ldots,0)\in U$ for all $w_1\in V_1$. Since $L$ is transitive on $[\ell]$, we conclude that $U=V$. Hence $V:H$ is primitive.
\end{proof}

\section{Orbits on spanning tuples}
\label{s:orbits}

 Recall that for non-negative integers $m$ and $d$ and a prime power $q$, the \textit{Gaussian binomial coefficient} is defined by
 $$ \binom{m}{d}_q:=\left\{
 \begin{array}{ll}
 \frac{(q^m-1)(q^{m-1}-1)\cdots (q^{m-d+1}-1)}{(q^d-1)(q^{d-1}-1)\cdots (q-1)} & \mbox{if} \ d\leq m\\
 0 & \mbox{if} \ d>m.
 \end{array}
 \right.
$$
Recall also that $\tbinom{m}{d}_q$ is precisely  the number of $d$-dimensional subspaces of $V_m(q)$.

\begin{proof}[Proof of Theorem $\ref{thm:orbits}$]
 Let $\mathcal{S}_{m,d}$ denote the set of spanning $m$-tuples of $\GL_d(q)$ on $V_d(q)=\mathbb{F}_q^d$ (viewed as row vectors). If $d>m$, then $\mathcal{S}_{m,d}$ is empty, so  we may assume that $d\leq m$. Let $\mathcal{M}_{m,d}$ denote the set of $(m\times d)$-matrices over $\mathbb{F}_q$ with rank $d$. There is a natural bijection between $\mathcal{S}_{m,d}$ and $\mathcal{M}_{m,d}$ defined by mapping $(v_1,\ldots,v_m)$ to the $(m\times d)$-matrix with rows $v_1,\ldots,v_m$. Now $\GL_d(q)$ acts on $\mathcal{M}_{m,d}$ by right multiplication, and the actions of $\GL_d(q)$ on $\mathcal{S}_{m,d}$ and $\mathcal{M}_{m,d}$ are equivalent under the bijection above, so the numbers of orbits of $\GL_d(q)$ on $\mathcal{S}_{m,d}$ and $\mathcal{M}_{m,d}$ are the same. 
 
 Let $\mathcal{O}_{m,d}$ be the set of orbits of $\GL_d(q)$ on $\mathcal{M}_{m,d}$. Define a map $\varphi$ from $\mathcal{O}_{m,d}$ to the set of $d$-dimensional subspaces of $\mathbb{F}_q^m$ (viewed as column vectors) by mapping the orbit $\{Ag:g\in \GL_d(q)\}$, for $A\in \mathcal{M}_{m,d}$, to the column space of $A$. Now $\varphi$ is well-defined since $A$ and $Ag$ have the same column space for all  $g\in \GL_d(q)$. Clearly $\varphi$ is surjective; it is also injective, for if $A$ and $B$ are elements of $\mathcal{M}_{m,d}$ with the same column space, then there exists $g\in \GL_d(q)$ such that $Ag=B$. Thus  $|\mathcal{O}_{m,d}|=\binom{m}{d}_q$, as desired.
\end{proof}

Next we give a simple estimation for the Gaussian binomial coefficient.

\begin{lemma}
\label{lemma:gaussianbinomial}
Let $d$ and $s$ be  positive integers and $q$ a prime power. Then 
 $$q^{ds}\leq \binom{d+s}{d}_q\leq \left(1-\frac{1}{q}-\frac{1}{q^2}\right)^{-1}q^{ds}.$$
\end{lemma}

\begin{proof}
 Observe that 
$$q^s\leq \frac{q^{s+i}-1}{q^i-1}=q^s\frac{q^i-\frac{1}{q^{s}}}{q^i-1}=q^s\frac{1-\frac{1}{q^{s+i}}}{1-\frac{1}{q^i}}$$
for  $1\leq i\leq d$. In particular, the lower bound holds. Moreover,   
$$\binom{d+s}{d}_q=q^{ds}\frac{\prod_{i=1}^d\left(1-\frac{1}{q^{s+i}}\right)}{\prod_{i=1}^d\left(1-\frac{1}{q^i}\right)}.$$
Since $\prod_{i=1}^d(1-1/q^{s+i})<1$ and $\prod_{i=1}^d(1-1/q^i)\geq 1-1/q-1/q^2$ by \cite[Lemma 3.5]{NeuPra1995}, the upper bound holds.
\end{proof}

\section{Base sizes of imprimitive linear groups}
\label{s:basesize}

We begin by establishing the first equality of Theorem \ref{thm:base size}.

\begin{lemma}
 \label{lemma:actual} 
Let $d$ and $\ell$ be  positive integers and $q$ a prime power. Let $V:=V_d(q)^\ell$ and $L\leq S_\ell$. Then 
$b_V(\GL_d(q)\wr L)=d+\min\{ s :\tbinom{d+s}{d}_q\geq d_{[\ell]}(L)\}.$
\end{lemma}

\begin{proof}
For $\GL_d(q)$ acting on $V_d(q)$, a multi-base of length $m$ is a spanning $m$-tuple for each positive integer $m$. Thus we may apply Theorem \ref{thm:orbits} and Proposition \ref{prop:wreath}. 
\end{proof}

Next we establish some bounds on the base size of $\GL_d(q)\wr L$; these we will use to determine the second equality of Theorem \ref{thm:base size}.

\begin{lemma} 
\label{lemma:bounds}
Let $d$ and $\ell$ be  positive integers and $q$ a prime power. Let $V:=V_d(q)^\ell$ and $L\leq S_\ell$. Then 
$$d+\left\lceil\frac{ \log c(q)d_{[\ell]}(L)}{d\log q} \right\rceil \leq b_V(\GL_d(q)\wr L)\leq d+\left\lceil\frac{ \log d_{[\ell]}(L)}{d\log q} \right\rceil $$
where $c(q):=1-\tfrac{1}{q}-\tfrac{1}{q^2}$.
\end{lemma}

\begin{proof}
By Lemmas \ref{lemma:gaussianbinomial} and \ref{lemma:actual}, 
$$b(\GL_d(q)\wr L)\leq d+\min\{ s :q^{ds}\geq d(L)\}=d+\left\lceil\tfrac{ \log d(L)}{d\log q} \right\rceil,$$
as desired. Similarly, again by Lemmas \ref{lemma:gaussianbinomial} and  \ref{lemma:actual}, 
$$b(\GL_d(q)\wr L)\geq d+\min\{ s :c(q)^{-1}q^{ds}\geq d(L)\}=d+\left\lceil\tfrac{ \log c(q)d(L)}{d\log q} \right\rceil,$$
as desired. 
\end{proof}

\begin{proof}[Proof of Theorem $\ref{thm:base size}$]
The first equality is Lemma \ref{lemma:actual}, so we focus on the second. For $(d,q)=(1,2)$, by Lemma \ref{lemma:actual}, 
$$b(\GL_1(2) \wr L)\geq 1+\min\{s:2^{s+1}\geq d(L) \}=1+\left\lceil\tfrac{ \log d(L)}{\log 2} \right\rceil -1,
$$
so by Lemma \ref{lemma:bounds}, the assertion of Theorem \ref{thm:base size} holds. For $(d,q)\neq (1,2)$, by Lemma \ref{lemma:bounds}, it suffices to show that $-1\leq \log c(q)/(d\log q)$, where $c(q)=1-1/q-1/q^2$. 
This is equivalent to proving that $q^d\geq c(q)^{-1}$. If $q\geq 3$, then $q^d\geq 3\geq 9/5\geq c(q)^{-1}$, and if $q=2$, then $d\geq 2$, so $q^d\geq 4=c(2)^{-1}$, as desired.
\end{proof}

\begin{proof}[Proof of Corollary $\ref{cor:base size}$]
Let $L$ be the subgroup of $S_\ell$ induced by $H$ on the set $\{V_1,\ldots,V_\ell\}$, and let $d:=\dim_{\mathbb{F}_q}(V_1)$. By Lemma \ref{lemma:irred},  the group $L$ is primitive and $(d,q)\neq (1,2)$. In particular, we may assume that $V_i=V_d(q)$ for $1\leq i\leq \ell$ and $H\leq \GL_d(q)\wr L$. By assumption, $L$ is not $S_\ell$ or $A_\ell$, so $d(L)\leq 4$ by Theorem \ref{thm:dist}. 
Since $\tbinom{d+1}{d}_q=q^d+q^{d-1}+\cdots + q+1\geq 4$, we may take  $s=1$ in Theorem \ref{thm:base size}. Thus $b(H)\leq d+1$.
\end{proof}

\section{Pyber's conjecture}
\label{s:Pyber}

We begin with a sufficient condition for Pyber's Conjecture in the case of affine groups with specified stabilisers.

\begin{lemma}
\label{L trick}
Let $d$ and $\ell$ be positive integers and $q$ a prime power.  Let $L\leq S_\ell$ and  $V:=V_d(q)^\ell$. Let $G_0$ be a group for which $\SL_d(q)^\ell\leq G_0\leq\GL_d(q)\wr L$ and suppose that $G_0$ induces the group $L$ on the $\ell$ factors of the decomposition of $V$. Let $G:=V:G_0$.
If $$ \log d_{[\ell]}(L)\leq C\log{|L|^{\tfrac{1}{\ell}}}+(C-1)\log{2}$$ for some absolute constant $C$ where $C>1$,  then  
$$b_V(G)\leq C \frac{\log{|G|}}{\log{n}}+C+2,$$
where $n:=q^{d\ell}$.
\end{lemma}

\begin{proof}
Let $X_0:=\GL_d(q)\wr L$ and   $a(d,q):=\prod_{i=1}^d(1-1/q^i)$.  Now $1\leq (1-1/q^i)q$ for $1\leq i\leq d$,  so $1\leq a(d,q)q^d$, which implies that $$0\leq C\log{a(d,q)}+Cd\log{q}.$$ 
Note that $(C-1)\log{2}\leq (C-1)d^2\log{q}$ since $C>1$. Thus by our assumption on $d(L)$,
$$
\log{d(L)}\leq  C\log{|L|^{\tfrac{1}{\ell}}}+(C-1)d^2\log{q}+ C\log{a(d,q)}+Cd\log{q}.
$$
 Dividing both sides by $d\log{q}$ and adding $d+2$, we obtain
\begin{equation}
\label{eqn:1}
\frac{\log{d(L)}}{d\log{q}}+d+2\leq  C\frac{\log{|L|}}{d\ell\log{q}}+Cd+C\frac{\log{a(d,q)}}{d\log{q}} +C +2
=C\frac{\log |X_0|}{d\ell\log q}+C+2
\end{equation}
since $|X_0|=\GL_d(q)^\ell|L|$ and $|\GL_d(q)|=q^{d^2}a(d,q)$.  Moreover, by Theorem \ref{thm:base size},
\begin{equation}
\label{eqn:2}
b(X_0)+1\leq \left\lceil\frac{\log d(L)}{d\log q}\right\rceil+d+1\leq \frac{\log d(L)}{d\log q}+d+2.
\end{equation}
Note that  $|G|=q^{d\ell}|G_0|$.  Since $L$ is the group  induced by the action of $G_0$ on the $\ell$ factors of the decomposition of $V$, and since the kernel of this action contains $\SL_d(q)^\ell$, it follows that $|X_0|=(q-1)^\ell|\SL_d(q)^\ell||L|\leq (q-1)^\ell|G_0|\leq |G|$. Thus by (\ref{eqn:1}) and (\ref{eqn:2}),
$$b(G)\leq b(G_0)+1\leq b(X_0)+1\leq C\frac{\log{|X_0|}}{d\ell\log{q}}+C +2\leq C\frac{\log{|G|}}{d\ell\log{q}}+C+2,$$
as desired.
\end{proof}

Recall that a permutation group $L\leq S_\ell$ is \textit{semiregular} if $b(L)=1$. 

\begin{proof}[Proof of Theorem $\ref{thm:Pyber}$]
 By Lemma \ref{L trick}, it suffices to find an absolute constant $C$ with $C>1$ such that \begin{equation}
\label{qd=2}
\log{d(L)}\leq C\log{|L|^{\frac{1}{\ell}}}+(C-1)\log{2}.
\end{equation}
 Note that for $\ell=1$, (\ref{qd=2}) holds with $C:=2$. 
 If (i) holds, then   (\ref{qd=2}) holds with $C:=\max\{2,\log(2c)/\log{2}\}$. Moreover, if (iii) holds, then   $d(L)\leq  2$, so  (\ref{qd=2}) holds with $C:=2$, and if (ii) holds and $L$ is  not $A_\ell$ or $S_\ell$, then $d(L)\leq 4$ by Theorem \ref{thm:dist}, so (\ref{qd=2}) holds with $C:=3$.

Observe that $k^k\leq k!^2$ for every positive integer $k$, for if $1\leq i\leq k$, then $k\leq i(k-i+1)$, so $k^k\leq \prod_{i=1}^k i(k-i+1)=k!^2$.

Suppose that (ii) holds and $L$ is $A_\ell$ or $S_\ell$, in which case $d(L)=\ell-1$ or $\ell$ respectively. Now $4(\ell-1)^\ell\leq 4(\ell-1)(\ell-1)!^2\leq \ell!^2$ and $\ell^\ell\leq \ell!^2$, so $d(L)^\ell\leq |L|^2$ in either case. Hence  (\ref{qd=2}) holds with $C:=2$.

If (iv) holds and $L=S_m\wr S_r$ where $\ell=mr$ and $m,r\geq 2$, then $d(L)\leq 2mr^{1/m}$ by Lemma \ref{lemma:chan}, and $(mr^{1/m})^{mr}\leq (m!^rr!)^2$, so  (\ref{qd=2}) holds with $C:=2$.
\end{proof}

\section*{Acknowledgements}

This research forms part of the Discovery Project grant DP130100106 of the second author, funded by the Australian Research Council. The first author is supported by that same grant. We would like to thank Aner Shalev for suggesting we look at the remaining open case for Pyber's conjecture, and Martin Liebeck for some observations regarding  Theorem \ref{thm:orbits}. 

We would especially like to give our heartfelt thanks to the group of mathematicians who, as a collective at the  2014 annual research retreat of the Centre for the Mathematics of Symmetry and Computation (CMSC), discovered and gave a recursive proof of a version of Theorem \ref{thm:orbits}. This group includes
Brian Corr, 
Alice Devillers, 
Stephen Glasby, 
Cai Heng Li, 
Dugald Macpherson and
Gabriel Verret.

\scriptsize
\bibliographystyle{acm}
\bibliography{jbf_references}

\begin{thebibliography}{10}

\bibitem{BaiCam2011}
{\sc Bailey, R.~F., and Cameron, P.~J.}
\newblock Base size, metric dimension and other invariants of groups and
  graphs.
\newblock {\em Bull. London Math. Soc. 43\/} (2011), 209--242.

\bibitem{Ben05}
{\sc Benbenishty, C.}
\newblock {\em On actions of primitive groups}.
\newblock PhD thesis, Hebrew University, Jerusalem, 2005.

\bibitem{BurSer15}
{\sc Burness, T.~C., and Seress, {\'A}.}
\newblock On {P}yber's base size conjecture.
\newblock {\em Trans. Amer. Math. Soc. 367\/} (2015), 5633--5651.

\bibitem{CamNeuSax1984}
{\sc Cameron, P.~J., Neumann, P.~M., and Saxl, J.}
\newblock On groups with no regular orbits on the set of subsets.
\newblock {\em Arch. Math. 43\/} (1984), 295--296.

\bibitem{Cha06}
{\sc Chan, M.}
\newblock The distinguishing number of the direct product and wreath product
  action.
\newblock {\em J. Algebr. Comb. 24\/} (2006), 331--345.

\bibitem{Dol2000}
{\sc Dolfi, S.}
\newblock Orbits of permutation groups on the power set.
\newblock {\em Arch. Math. 75\/} (2000), 321--327.

\bibitem{Faw2013}
{\sc Fawcett, J.~B.}
\newblock The base size of a primitive diagonal group.
\newblock {\em J. Algebra 375\/} (2013), 302--321.

\bibitem{GluMag1998}
{\sc Gluck, D., and Magaard, K.}
\newblock Base sizes and regular orbits for coprime affine permutation groups.
\newblock {\em J. London Math. Soc. 58\/} (1998), 603--618.

\bibitem{LieSha1999}
{\sc Liebeck, M.~W., and Shalev, A.}
\newblock Simple groups, permutation groups, and probability.
\newblock {\em J. Amer. Math. Soc. 12\/} (1999), 497--520.

\bibitem{LieSha2002}
{\sc Liebeck, M.~W., and Shalev, A.}
\newblock Bases of primitive linear groups.
\newblock {\em J. Algebra 252\/} (2002), 95--113.

\bibitem{LieSha2014}
{\sc Liebeck, M.~W., and Shalev, A.}
\newblock Bases of primitive linear groups {II}.
\newblock {\em J. Algebra 403\/} (2014), 223--228.

\bibitem{NeuPra1995}
{\sc Neumann, P.~M., and Praeger, C.~E.}
\newblock Cyclic matrices over finite fields.
\newblock {\em J. London Math. Soc. 52\/} (1995), 263--284.

\bibitem{Pyb1993}
{\sc Pyber, L.}
\newblock Asymptotic results for permutation groups.
\newblock {\em DIMACS Ser. Discrete Math. Theoret. Comp. Sci. 11\/} (1993),
  197--219.

\bibitem{Ser1996}
{\sc Seress, {\'A}.}
\newblock The minimal base size of primitive solvable permutation groups.
\newblock {\em J. London Math. Soc. 53\/} (1996), 243--255.

\bibitem{Ser1997}
{\sc Seress, {\'A}.}
\newblock Primitive groups with no regular orbits on the set of subsets.
\newblock {\em Bull. London Math. Soc. 29\/} (1997), 697--704.

\bibitem{Ser2003}
{\sc Seress, {\'A}.}
\newblock {\em Permutation group algorithms}.
\newblock Cambridge University Press, Cambridge, 2003.

\end{thebibliography}

\end{document}